\documentclass[12pt, a4paper, reqno]{amsart}
\usepackage{eucal}
\usepackage{amsmath}%
\usepackage{amsfonts}%
\usepackage{amssymb}%
\usepackage{graphicx}

\newcommand{\N}{\mathbb N}
\newcommand{\C}{\mathbb C}
\newcommand{\R}{\mathbb R}
\newcommand{\I}{\mathcal I}
\newcommand{\an}[1]{a_{#1}}
\newcommand{\V}[1]{b_{#1}}
\newcommand{\ap}[1]{a^{\textsc{per}}_{#1}}
\newcommand{\Vp}[1]{b^{\textsc{per}}_{#1}}
\newcommand{\p}{^{\textsc{per}}}
\newcommand{\inv}{^{-1}}
\newcommand{\vect}[2]
 {\left(\begin{array}{@{}c@{}}#1\\
                             #2\end{array}\right)}
\newcommand{\matr}[4]
 {\left(\begin{array}{@{}cc@{}}
 #1&#2\\
 #3&#4
 \end{array}\right)}
\newcommand{\df}\equiv%{\stackrel{\mbox{\scriptsize{Def}}}{=} }
\newcommand{\J}{\mathbf J}
\newcommand{\mT}{T}
\newcommand{\mU}{U}
\newcommand{\mL}{\Lambda}
\newcommand{\mI}{ {\mathbb I}}
\newcommand{\mW}{W}
\newcommand{\mP}{\Phi}
\newcommand{\set}[1]{\left\{#1\right\}}
\newcommand{\sett}[2]{\left\{#1\big|#2\right\}}
\newcommand{\tr}{\mathrm{Trace}\,}
\newcommand{\ac}{\mathrm{ac}}
\newcommand{\im}{\mathrm{Im}\,}
\newcommand{\re}{\mathrm{Re}\,}

\newcommand{\spr}[2]{\left\langle\,#1,#2\,\right\rangle}
\newcommand{\choice}[4]{\left\{\begin{array}{rl}#1,\; & #2\\#3, &#4\end{array}\right.}

\newcommand{\bo}[1]{\mathbf{#1}}

\newtheorem{theorem}{Theorem}
\newtheorem*{theorem*}{Theorem}
\theoremstyle{plain}

\newtheorem{lemma}{Lemma}
\newtheorem{prop}{Proposition}
\numberwithin{equation}{section}

\newcommand{\spec}[1]{\sigma(#1)}

\newcommand{\ddo}{\Delta(\zeta)}

\newcommand{\oo}{$\mathrm{(}$}
\newcommand{\cc}{$\mathrm{)}\,$}
\begin{document}
\title[Preservation of a.c.\ spectrum]{
Preservation of absolutely continuous spectrum of periodic Jacobi
operators under perturbations of square--summable variation}

\author{U. Kaluzhny}
\author{M. Shamis}
\address{Institute of Mathematics,\newline \indent The Hebrew University, \newline
\indent 91904 Jerusalem, Israel.}%
\email{kaluzhny@math.huji.ac.il, shamis@math.huji.ac.il}%

\begin{abstract}
We study self-adjoint bounded Jacobi operators of the form:
\[
(\J\psi)(n) = \an{n}\psi(n + 1) + \V{n}\psi(n) + \an{n-1}\psi(n - 1)
\]
on $\ell^2(\N)$.
We assume that
 for some fixed $q\in\N$, the $q$-variation of $\{a_n\}$ and $\{b_n\}$ is square-summable and
  $\{a_n\}$ and $\{b_n\}$ converge to $q$-periodic sequences $\{a_n\p\}$ and $\{b_n\p\}$, respectively.

Our main result is that under these assumptions  the essential
support of the absolutely continuous part of the spectrum of $\J$ is
equal to that of the asymptotic periodic Jacobi operator.

This work generalizes a recent result of S.~A.~Denisov.
\end{abstract}
%\date{March 15, 2002}
\maketitle

%\footnotetext[]{47B36, 47A55}

\section{Introduction}
\label{sec:Introduction}
In this work we consider bounded self-adjoint Jacobi operators on $\ell^2(\N)$ with a Dirichlet
boundary condition, i.e., discrete one-dimensional operators  defined by:
\begin{equation}\label{op}
\begin{array}{lcl}
(\J \psi)(n) &=& \an{n}\psi(n + 1) + \V{n}\psi(n) + \an{n-1}\psi(n - 1),\;n>1 \\
(\J \psi)(1) &=& \an{1}\psi(2) + \V{1}\psi(1),
\end{array}
\end{equation}
where $\bo a = \{\an{n}\}$ and $\bo b = \{\V{n}\}$ are bounded sequences of real numbers, $\inf \an{n}> 0$, and for some fixed $q\in\N$ the sequences $\bo a$ and $\bo b$ are of square-summable $q$-variation, namely:
\begin{equation}\label{s2}
\sum_{n=1}^\infty |\an{n + q} - \an{n}|^2 + |\V{n + q} - \V{n}|^2 < \infty.
\end{equation}
We also assume that the limits
\begin{equation}\label{per}
\ap{k}\df\lim_{n\rightarrow\infty}\an{k + nq}, \;
 \Vp{k}\df\lim_{n\rightarrow\infty}\V{k + nq}
\end{equation}
exist for all $k\in\N$.

From (\ref{s2}) and (\ref{per}) one can see that $\{a\p_n\}$ and $\{b\p_n\}$ are periodic sequences with period $q$, i.e.,
\[
 \ap{k} = \ap{k+q},\;\;  \Vp{k} = \Vp{k+q},
\]
for all $k$.
Denote by  $\J \p$ the periodic Jacobi operator defined by
$\{a\p_n\}$ and $\{b\p_n\}$.
Obviously, $\ap{n}\neq0$ for all $n$.

For Jacobi operators $\J$ (with $\inf_n a_n > 0$), we define $\Sigma_\mathrm{ac}(\J)$, the essential
support of the absolutely continuous spectrum of $\J$, as the equivalence class, up to sets of zero Lebesgue measure, of the set
\[ \left\{ E \in \mathbb{R} \, \big| \,
    \lim_{\epsilon \searrow 0} \im \langle \delta_1, \, (\J - E - i\epsilon)^{-1} \delta_1 \rangle
    \quad \text{exists and differs from 0.} \right\}. \]
Here $\langle \cdot, \, \cdot \rangle$ denotes the scalar product in $\ell^2(\mathbb{N})$ and $\delta_1 = (1,0,0,\ldots)^T$.
%The absolutely continuous spectrum of $\J$, $\sigma_\mathrm{ac}(\J)$, is the essential closure of $\Sigma_\mathrm{ac}(\J)$, namely,
%\[ \sigma_\mathrm{ac}(\J) = \left\{ E \in \mathbb{R} \, \big| \,
%    \forall \epsilon > 0, \big| \Sigma_\mathrm{ac}(\J)
%        \cap (E - \epsilon, E + \epsilon) \big| > 0 \right\}, \]
%where $ | \cdot | $ denotes Lebesgue measure.
In what follows, equalities
of the form $\Sigma_\mathrm{ac}(\J) = S$ should be  understood as
$ (\Sigma_\mathrm{ac}(\J) \setminus S) \cup  (S \setminus\Sigma_\mathrm{ac}(\J))$  being a set of zero Lebesgue measure.

Our main result in this paper is the following:
\begin{theorem}\label{mainTheorem} Under the assumptions \oo\ref{s2}\cc and \oo\ref{per}\cc
\begin{equation}\label{main_eq}
\Sigma_{\ac}(\J ) = \Sigma_{\ac}(\J \p).
\end{equation}
\end{theorem}
\noindent This theorem generalizes a recent result of Denisov  \cite{Den}, who proved the case $\bo a \equiv 1$, $\lim b_n = 0$, and settles the full  \cite[Conjecture 2]{KL}.

Preservation of absolutely continuous spectrum under decaying perturbation has been intensively studied during the last two decades.

Let us denote by $\bo J(\bo a,\bo b)$ the Jacobi operator defined in (\ref{op}) and consider a Jacobi operator of the form
$\J(\bo a +{\alpha}, \bo b + \beta)$, where
${ \alpha}=\{\alpha_{n}\}_{n=1}^\infty$ and $ \beta=\{\beta_{n}\}_{n=1}^\infty$ are decaying sequences.

Recall the following classical fact: from Weyl's Theorem (see, for
example, \cite[Theorem XIII.14]{RS}), one obtains that the essential
spectrum of any Jacobi operator is preserved under a decaying
perturbation. The absolutely continuous spectrum, on the other hand,
is much easier to destroy. Indeed, Last \cite{L_destr} has
constructed an example of a Jacobi operator $\bo J(\bo1,{\bo b} +
\beta)$ with $\lim_{n\to\infty}{b}_n = \lim_{n\to\infty}\beta_n =
0$, so that $\beta$ is of summable 1-variation (namely, $\sum
|\beta_{n+1}-\beta_n| < \infty$), both $\bo J(\bo1, {\bo b})$ and
$\bo J(\bo1,\beta)$ have purely a.c.\ spectrum on $(-2,2)$ with
essential support $(-2,2)$, but $\bo J(\bo1,{\bo b} + {\beta})$ has
empty absolutely continuous spectrum. In particular, adding a
decaying perturbation of summable $1$-variation to a Jacobi operator
can fully ``destroy'' its absolutely continuous spectrum.

Numerous works have been devoted to the following question: which conditions on $(\bo a, \bo b)$ and on the perturbation $(\alpha, \beta)$ ensure that the essential support of the absolutely continuous spectrum of the perturbed operator $\J(\bo a +{\alpha}, \bo b + \beta)$ coincides with that of $\J(\bo a, \bo b)$?
For recent reviews on the subject see, e.g., the articles by Denisov and Kiselev \cite{DK} and by Killip \cite{Kil_rev}.

In particular, in the 1990's much work has been done towards showing that square-summable perturbations of the free Laplacian do not change its a.c.\ spectrum. For summable perturbations, the analogous result has been known at least from
the 1950's and follows, in particular, from Birman-Kato theory of trace class perturbations (see \cite[vol III]{RS}). The passage to $\ell^2$ perturbations took a lot of effort (see, e.g., \cite{CKR, CK, DK, Kil, Kis, KLS, Rem}). Eventually, Killip and Simon  \cite{KS} (strengthening Deift and Killip \cite{deift_killip}) proved the following.
\begin{theorem*}[\textbf{Killip--Simon}]
\label{th:KS}
    If a perturbation  is \emph{square-summable}, that is
\[\sum_{n=1}^\infty |a_n - 1|^2 + |b_n|^2 <\infty,\]
  then $\Sigma_{ac}(\bo J(\bo a,\bo b))$, the
essential support of the a.c.\ spectrum of $\bo J(\bo a,\bo b)$, is equal to $[-2, 2]$.
\end{theorem*}

An analog of this theorem may hold for a wider class of operators (see the conjecture in \cite[the remark after Theorem 1.6]{KLS1}). Killip \cite{Kil} has proven it for the case of $\bo a = 1$ and $\bo b$ being a periodic sequence plus a square--summable perturbation.

The $\ell^2$ condition is known to be sharp: the works of Delyon, Simon and Souillard  \cite{Del,DSS,Si82} and Kotani-Ushiroya  \cite{Kot, KU} on decaying random potentials
showed in the 1980's that perturbations that are not
square--summable can result in purely singular spectrum.
 Later on, Simon  \cite{Si95} proved that
for any $p>2$ the potentials $\{\beta_{n}\}\in\ell^{p}$ resulting in purely singular continuous spectrum form a dense $G_{\delta}$ set in the topology of $\ell^{p}$. Therefore, the a.c.\ spectrum of any Jacobi operator can be destroyed by adding a perturbation in $\ell^p$ for any $p>2$.

The above mentioned results exploit the decay of the
perturbation. Another, weaker, criterion for a perturbation to
preserve the a.c.\ spectrum has been studied for a long time.
Weidmann \cite{Wei} in 1967 proved that the a.c.\ spectrum is
preserved under a perturbation of  summable $1$-variation (Weidmann
actually proved a variant of this for continuous Schr\"{o}dinger
operators. For a proof of the discrete case, see Dombrowski-Nevai
\cite{DoNe} or Simon \cite{Si96}.)

Golinskii and Nevai \cite{GN} generalized Weidmann's result to
arbitrary $q \geq 1$; their argument, written originally for
orthogonal polynomials on the unit circle, yields the following
result for Jacobi operators:
\begin{theorem*}[\textbf{Golinskii--Nevai}]
\label{tgn}
    Let $\bo J(\bo a,\bo b)$ be a periodic Jacobi operator of period $q$
and let $\{\alpha_{n}\}$ and $\{\beta_{n}\}$ be decaying sequences of summable $q$-variation, namely
\begin{equation}\label{eq:q}
   \sum_{n=1}^\infty |\alpha_{n+q}-\alpha_n| + |\beta_{n+q}-\beta_n|<\infty.
\end{equation}
Then the essential support of the
a.c.\ spectrum of $\bo J(\bo a + \bo {\alpha},\bo b+\bo {\beta})$ is equal to the
spectrum of $\bo J(\bo a,\bo b)$ and, moreover, the spectrum of $\bo J(\bo a + \bo {\alpha},\bo b+\bo {\beta})$
is purely absolutely continuous on the interior of the bands that make up the spectrum of $\bo J(\bo a,\bo b)$.
\end{theorem*}

Thus, Theorem \ref{mainTheorem} extends the Golinskii--Nevai
theorem, showing that the $\ell^1$ condition (\ref{eq:q}) can be
relaxed similarly to what has been done in the case of a summable
perturbation. As we already mentioned, one cannot weaken it to
$\ell^p$ for $p>2$. The condition (\ref{per}) on the limiting
behavior of $\bo a$ and $\bo b$ is also natural. In the case of a
summable $q$-variation it is obviously implied. If dropped in our
case, the a.c.\ spectrum of the operator cannot be determined by the
limiting behavior of $\bo a$ and $\bo b$ only. Indeed, Last (private
communication) has recently constructed examples showing, in
particular, that \cite[Conjecture 9.5]{BLS} is false (except, maybe,
for $q=1$).

A special case of Theorem \ref{mainTheorem}, with $\J(\bo a,\bo b)= \J(\bo1,\bo0)$ and $\alpha \equiv 0$, was conjectured by Last  \cite[Conjecture 1.6]{L_destr}. A variant of this conjecture for the special case $q = 1$
has also been made by Simon  \cite[Chapter 12]{SiOP}.
A notable result in this direction has been obtained by
Kupin  \cite{Kup} who showed that the essential support of
the a.c.\ spectrum of $\J(\bo1, \bo0)$ is still preserved if a decaying sequence  $\beta$ of a square-summable variation (with $q=1$) is added to $\bo b$, under an additional restriction that $\beta \in \ell^{m}$ for some $m\in \N$.
In general, more tools exist to explore spectra of decaying perturbations of $\J(\bo1, \bo0)$, e.g., sum rules and trace formulas used, in particular by Kupin.  Extensions of these methods to general periodic operators are sometimes quite involved. A recent work of Damanik, Killip and Simon  \cite{DKS} has given some definitive results in this direction. Moreover, methods used by Kupin impose very restrictive conditions on the behavior of $\alpha$ and $\beta$.

Kim and Kiselev  \cite{KK} have relaxed the conditions on $\alpha$ and $\beta$ using an analog of the techniques developed by Christ and Kiselev  \cite{CK} for continuous (namely, differential) Schr\"odinger operators.
The results of Kim and Kiselev imply that if $\bo a = \bo 1$, $\bo b = \bo 0$, $\alpha = \bo 0$ and  $\beta$ is a bounded (but not necessarily decaying) sequence obeying $\sum_{n=1}^\infty |\beta_{n+1}-\beta_n|^p<\infty$ for some $p < 2$, then the essential support of the a.c.\ spectrum coincides with  $[-2+\limsup\beta_n, 2+\limsup\beta_n]\cap[-2+\liminf\beta_n, 2+\liminf\beta_n]$. In the special case $\lim \beta_n = 0$ this implies $\Sigma_\ac(\J(\bo 1, \beta)) = [-2,2] = \Sigma_\ac(\J(\bo 1,\bo 0))$.
%We, however, do not know how to extend their techniques  to the case $p=2$.

Very significant progress has been recently made by Denisov  \cite{Den}, who proved the full  \cite[Conjecture 1.6]{L_destr}. To analyze the a.c.\ part of the spectral measure of a perturbed
operator, Denisov  \cite{Den} uses an important factorization of
the Radon-Nikodym derivative of the spectral measure. This
expression, involving the Jost function,  has been proven by Killip
and Simon  \cite{KS} for trace class
perturbations of the free Laplacian.  We do not know how to extend the technique of perturbation determinants from  \cite{KS}
to the case of periodic Jacobi operators.
However, approximating the operator under
consideration by finitely supported perturbations, combined with a number
of ideas of Denisov  \cite{Den}, we are able to prove
Theorem \ref{mainTheorem}  by a more elementary
technique.

For a finitely supported perturbation $J^N$ and an energy $\zeta = E+i\epsilon$
in the upper half-plain of $\C$ we can explicitly construct an $\ell^2(\N)$
solution of $J^Nu^N(\zeta) =\zeta u^N(\zeta)$. As $\epsilon\searrow0$,
the vector $u^N(\zeta)$ becomes a Jost solution, that is,
asymptotically a Bloch wave.  Using $u^N_0(\zeta)$ similarly to
the way it is done in  \cite{KS}, we are able to
express the Radon--Nikodym derivative of the spectral measure as a
product of a well-behaving function depending on the parameters
of the periodic operator and a harmonic function $u^N_0(\zeta)$
depending on the perturbation, which can be called, by analogy
with  \cite{KS}, the Jost function.

As we show below, to prove that the a.c.\ spectrum of
$\J $ fills $\spec{\J \p}$ it will be sufficient to bound the integrals
$\int_I \ln |u^N_0(E)| dE$ uniformly in $N$, when $I\subset \spec{\J \p}$
is an interval of a special kind. The behavior of $u^N_0$ on $\R$
is difficult to analyze, so the estimates of the above mentioned
integrals  are obtained by controlling the behavior of $u^N_0$ off the real line.

The rest of the paper is organized as follows. In Section 2 we give
an outline of the proof. Section 3 contains the derivation of the
expression of the Radon--Nikodym derivative of the spectral
measure for the case of a truncated perturbation. In Section 4 we
analyze the Jost function. Section 5 concludes the
proof of the Main Theorem.

We would like to thank Y.~Last and S.~Sodin for pleasant and useful discussions.
This research was supported in part
by The Israel Science Foundation (Grant No.\ 1169/06) and by Grant 2006483
from the United States-Israel Binational Science Foundation (BSF), Jerusalem, Israel.

\section{Outline of the proof}
Given an operator $\J $ as above,  we consider, for $E \in \R$, the
associated eigenvalue equation
\begin{equation}\label{eigv}
\J \psi = E\psi,
\end{equation}
where $\psi$ is a sequence of complex numbers. For $n\in\N$, define the
one-step transfer operators associated with (\ref{eigv}) by
\begin{equation}\label{tr1_matr}
    \mT_n(E) =
    \matr {\frac{E - \V{n}}{\an{n}}}{-\frac{\an{n-1}}{\an{n}}}10:
    \vect{\psi(n)}{\psi(n-1)} \longmapsto \vect{\psi(n+1)}{\psi(n)}.
\end{equation}
(with $a_0=a\p_0$).

For  $q$ from (\ref{s2}) and $n\geq 0$, denote by $\mP_n(E)$  the
following product of $q$ consecutive one-step transfer matrices:
\begin{equation}\label{tr_period}
\mP_n(E)\!=\! \mT_{(n+1)q}(E)\cdots\mT_{nq + 2}(E) \mT_{nq+1}(E)\!
=\! \matr{A_n(E)}{B_n(E)}{C_n(E)}{D_n(E)}\!.
\end{equation}
It is easy to see that $A_n$, $B_n$, $C_n$ and $D_n$ are polynomials
of degrees $q$, $q-1$, $q-1$ and $q-2$, respectively, with non-zero
leading term and real coefficients.

For $\J \p$, the matrix $\mP\p_n(E)$ does not depend on $n$, so we write:
\begin{equation}\label{tr_periodic}
\mP\p(E) = \matr{A(E)}{B(E)}{C(E)}{D(E)}.
\end{equation}

In what follows we use the following well known fundamental fact
(see, e.g.,  \cite{L} for a fuller exposition):
 for periodic operators the spectrum is equal to
\begin{equation}\label{delta}
\spec{\J \p} = \sett{E}{|\tr \mP\p(E)| \leq 2},
\end{equation}
and it is purely absolutely continuous.

Let $\mu$ be the spectral measure of $\J $ with respect to the cyclic vector $\delta_1$, that is:
\[
\spr{\delta_1}{(\J -E)\inv\delta_1} = \int\frac{d\mu(x)}{x-E}.
\]
From Weyl's theorem  \cite{RS}, we know that $\sigma_{\mathrm{ess}}(\J ) =
\sigma_{\mathrm{ess}}(\J \p)$. We show that the a.c.\ spectrum of
$\J $ fills $\spec{\J \p}$  (i.e., $|\spec{\J\p} \setminus
\Sigma_{\mathrm{ac}}(\J)| = 0$) by choosing a family of intervals
covering almost every point of $\spec{\J \p}$ and proving that
\[
\int_{I}{\ln \frac{d\mu_{\ac}}{dE}(E)\; dE} >-\infty,
\]
for any interval $I$ in this family. Here $\mu_{\ac}$ is the a.c.\ part of the spectral measure $\mu$.

The main idea of our proof is to approximate $\J $ by truncated
perturbations of $\J \p$. Namely, for $N\in \N$ we define $\J ^N$ to be the Jacobi operator with
\begin{equation}\label{trunc}
    a_n^N = \choice{a_n}{n< (N-1)q}{a_n\p}{n\geq (N-1)q}\;
    \V n^N = \choice{\V n}{n\leq (N-1)q}{\Vp n}{n> (N-1)q}\!\!.
\end{equation}

Let $\mu^N$ be the spectral measure of $\J ^N$ with respect to the
same cyclic vector $\delta_1$. Since, as $N\rightarrow\infty$, the
coefficients defined by (\ref{trunc}) converge to the Jacobi
parameters of $\J \p$, the measures $\mu^N$ converge weakly to the
spectral measure $\mu$ for $\J $. Hence, we can use the
semi-continuity of entropy. Namely, we use the following
\begin{lemma}[{ \cite[Corollary 5.3]{KS}}]\label{entropy}
Suppose a sequence of absolutely continuous measures $\set{\rho_n(x)dx}$  converges weakly to a measure $\rho(x)dx$.
Then for any measurable set $S$,
\[\int_S \ln \rho(x) dx \geq \liminf_{n\rightarrow\infty}\int_S \ln \rho_n(x) dx.
\]
\end{lemma}
For $\zeta = E +i\epsilon, \epsilon >0$, the operator $G^N(\zeta)=
(\J ^N - \zeta)\inv$ is well-defined and bounded since
$\zeta\notin\sigma(\J ^N)$. For
\[G^N(1,1;\zeta)=\spr{\delta_1}{G^N(\zeta)\delta_1},\]
it is known that
\begin{equation}
\label{green}
\im G^N(1,1;E+i0)\df\lim_{\epsilon\searrow0} \im G^N(1,1;E+i\epsilon) = \pi\frac{d\mu^N_{ac}}{dE}(E).
\end{equation}
Since $(\J ^N - \zeta)\times G^N(\zeta)= \mI$, the vector $u = G^N\delta_1$ solves the system
\begin{equation}\label{greensys0}
\left\{\begin{array}{lcl}
\an{1}^Nu_2 + (\V1^N - \zeta)u_1& = &1,\\
\an{n}^Nu_{n + 1}+(\V n^N - \zeta)u_n + \an{n-1}^Nu_{n - 1}& =& 0,
\;\mbox {for every $n\geq 2$}.
\end{array}\right.
\end{equation}

To solve the above system, we slightly modify it, considering
instead   $\ell^2(\{0\}\cup\N)$, and find a vector $u^N(\zeta)\in\ell^2(\{0\}\cup\N)$ satisfying
\begin{equation}\label{greensys}
\an{n}^Nu^N_{n + 1}+(\V n^N - \zeta)u^N_n + \an{n-1}^Nu^N_{n - 1} = 0,
\;\mbox {for every $n\geq 1$},
\end{equation}
defining $\an0^N = \ap0.$ For $u^N$ a solution of (\ref{greensys}),
the vector $\frac{u^N(\zeta)}{\ap0u^N_0(\zeta)}$ solves (\ref{greensys0}), hence
\begin{equation}
\label{Gu}
G^N(1,1;\zeta)= \frac{-u^N_1(\zeta)}{\ap0u^N_0(\zeta)},
\end{equation}
provided $u^N_0(\zeta) \neq 0$. Indeed, $u^N_0(\zeta) = 0$ implies
$(\J ^N-\zeta)u=0$, meaning $\zeta\in\spec{\J ^N}$, which is false
for $\epsilon \neq 0$. However, we still need to prove that
$\lim_{\epsilon\searrow 0}u^N_0(E+i\epsilon) \neq 0$.

In the next section we solve (\ref{greensys}) and obtain the key
formula (\ref{final}) of the form
\[\frac{d\mu^N_{ac}}{dE}(E) =
\frac{ F(E)}{| u^N_0(E)|^2},
\]
where $F(E)$ is a function that depends on the Jacobi parameters of
$\J\p$ only. As will be shown, this function is well-behaving on
$I$, therefore it does not affect the estimate on the entropy.
Hence,  a lower bound on $\int_I\ln\frac{d\mu^N_\ac}{dE} (E) \;dE$
will follow from an upper bound on $\int_I\ln |u^N_0(E)| \;dE$.  We
do not obtain pointwise estimates for $u^N_0$, but rather estimate
the integral $\int_I\ln |u^N_0(E)| \;dE$ using, following Killip
 \cite{Kil} and Denisov  \cite{Den}, the
following
\begin{lemma}[{ \cite[Lemma \textbf{A}.3]{Den}}]\label{denisov_harmonic_lemma}
Assume $f(\zeta)$ is harmonic on
\[\Pi = \sett{\zeta = x + iy \in\C}{a < x < b, 0 < y < c },\]
continuous on the closure of ${\Pi}$, and for some $C, \alpha > 0$,
\begin{equation*}
\begin{array}{ll}
    \int_{a}^{b}{f^+(x)dx} < C,&\mbox{where } f^+= \max(f, 0), \;
        f^- = f^+ - f,\\ % switched to standard notation for f^-
    f(\zeta) > - C y^{-\alpha},& \mbox{for }\zeta\in{\Pi},\\
    f(\zeta) < C, & \mbox{for }y > C (1 + \alpha)^{-1} > 0.
\end{array}
\end{equation*}
Then there exists a constant $B>0$, depending on $C$ and $\alpha$, so that
\begin{equation*}
\int_a ^b{f^{-}(x)dx} < B.
\end{equation*}
\end{lemma}
Note that although the statement is slightly different from
 \cite{Den}, the argument of Denisov proves it as well.

As one can see from Lemma~\ref{denisov_harmonic_lemma}, we need to
provide estimates for $\ln |u_0^N(\zeta)|$ off the real line. The
value of $u^N_0(\zeta)$ comes from the transfer matrices of $\J^N$
applied to the fixed boundary condition. Below we show how the
behavior of transfer matrices can be analyzed and an estimate on
$u^N_0(\zeta)$ can be obtained when the matrices have an eigenvalue
with absolute value greater than one off the real line. This can be
guaranteed, though, only if the determinant of the transfer matrices
is equal to one (Proposition~\ref{lbound} of the Appendix). However,
in our case, $\det \mP_n = a_{nq}/a_{(n+1)q}\neq 1$ because of the
perturbation, and we ``fix'' the situation by a renormalization of
the transfer matrices.

The semi-continuity of the entropy (Lemma~\ref{entropy}) finishes
the proof.

\section{Truncated perturbations}
\label{sec:TruncatedPertrubations} In this section we obtain an
explicit formula connecting $\mu^N_\ac$, the a.c.\  part of the
spectral measure of $\J ^N$, to its Jacobi parameters.

Denote $\Delta(E) = \tr \mP\p(E)$.
Note that:
\begin{itemize}
    \item $\Delta(E)$ is a polynomial of degree $q$ with real coefficients.
    \item Since $\det \mP_n(\zeta) = \frac{a_{nq}}{a_{(n+1)q}}$, for the periodic case $\det \mP\p(\zeta) = 1$.
\end{itemize}

Consider the set $S = \{ E\in\R \big| |\Delta(E)| < 2,
\Delta'(E)\neq 0, C(E)\neq 0 \}$, where $C(E)$ is defined by (\ref{tr_periodic}).
For a closed interval $I\subset S$ denote
$I_{\epsilon_I} = \sett{E+i\epsilon}{E\in I, 0 <
\epsilon\leq\epsilon_I}$. We choose a collection $\I$ of closed
intervals $I\subset S$ and $\epsilon_I$  small enough for each $I$,
such that $S = \bigcup_{I\in\I}I$ and for every $I\in\I$ the following conditions hold true:
\begin{itemize}
    \item $\Delta$ is invertible on $I_{\epsilon_I}\cup I$,
    \item $C(\zeta)\neq0$ for any $\zeta\in I_{\epsilon_I}$,
    \item For any $\zeta\in I_{\epsilon_I}$, the matrix $\mP\p(\zeta)$ has an eigenvalue
        \[ z(\zeta) = \frac{\Delta(\zeta) + \sqrt{\Delta^2(\zeta) - 4}}{2}, \]
        such that  $|z(\zeta)| < 1$ as ensured by Proposition \ref{lbound} from the Appendix.
        We continue $z(\zeta)$ to $I$ as well; $z(\zeta)$ is
        analytic in $I \cup I_{\epsilon_I}$.
\end{itemize}
We choose the collection $\I$ so that $\sum_{I\in\I} |I| < \infty$, which is possible, since $\spec{\J \p}$ is a union of $q$ closed intervals and the previous conditions do not hold on a finite number of points only.
    Then, to prove Theorem \ref{mainTheorem}, it will be sufficient to prove that for any $I\in
\I$ the a.c.\ spectrum of $\J $ fills $I$.

Therefore, we fix $I\in\I$ and prove that there exists a uniform bound $C$, so that for any $N\in\N$,
\[
\int_I\ln\frac{d\mu^N_{\ac}}{dE}(E)\; dE \geq - C.
\]
%We define a continuous function $f$ on $I_{\epsilon_I}\cup I$
% by $f(\zeta) = \Delta\inv(\overline{\Delta(\zeta)})$. Then, for any $\zeta\in I_{\epsilon_I}\cup I$,
%    \[\Delta(f(\zeta)) = \overline{\Delta(\zeta)},
%\]
%and, since $\Delta$ is a real polynomial,
%\begin{equation}\label{fe}
%    f(E) = E,\; \mbox{for all }E\in I.
%\end{equation}
%From the definition of $\Delta(\zeta)$:
%\begin{equation}
%\label{sum_eq_sum}
%\Delta(f(\zeta)) = z(f(\zeta)) + z\inv(f(\zeta))=
%\overline{\Delta(\zeta)} = \overline{z(\zeta)} +
%\overline{z\inv(\zeta)}.
%\end{equation}
%Since on $I_{\epsilon_I}$ $|z(f(\zeta))\cdot\overline{z(\zeta)}|=
%|z(f(\zeta))\cdot{z(\zeta)}| < 1$ by the definition of $z$,
%(\ref{sum_eq_sum}) means that
%\begin{equation}
%\label{zf}
%z(f(\zeta)) = \overline{z(\zeta)},
%\end{equation}
%which is true for $\zeta \in I$ as well, since $z(\cdot)$ is continuous.

Let $\vect {x(\zeta)}{y(\zeta)}$ be an eigenvector of $\mP\p(\zeta)$ corresponding to the eigenvalue $z(\zeta)$. Recall that
\[\mP\p(E) = \matr {A(E)}{B(E)}{C(E)}{D(E)},\]
where $A,B,C,D$ are polynomials with real coefficients, and
$C(E)\neq0$ on $I$. Hence we can take
\[
x(\zeta) = z(\zeta) - D(\zeta)\;\;\mbox{ and }\;\;y(\zeta) = C(\zeta).
\]

Now we set
\begin{equation}
\label{udef}
\vect {u^N_{Nq+1}(\zeta)}{u^N_{Nq}(\zeta)} = \vect{x(\zeta)}{y(\zeta)} =
\vect{z(\zeta) -D(\zeta)}{C(\zeta)}
\end{equation}
and calculate $u^N_n(\zeta)$ for any $n$, solving the recursion (\ref{greensys}) backward and forward. In this way we have built a solution of (\ref{greensys}) which lies in $\ell^2(\{0\}\cup\N)$. Indeed, it is easy to see from (\ref{trunc}) that for $n> Nq$,
\[u^N_{n+q}(\zeta) = z(\zeta)u^N_{n}(\zeta),\]
 so $|z(\zeta)| < 1$ ensures that $u$ decays exponentially.

%In what follows we will need another vector in
%$\ell^2(\{0\}\cup\N)$: a solution of the conjugate system
%\begin{equation}\label{vsys}
%\an{n}^Nu^N_{n + 1}+(\V n^N - f(\zeta))u^N_n + \an{n-1}^Nu^N_{n - 1} = 0,
%\;\mbox {for every $n\geq 1$},
%\end{equation}
%where $f$ is as defined above. In the same way as for $u^N(\zeta)$, we take
%\begin{equation}
%\label{vdef}
%\vect {v^N_{Nq+1}(\zeta)}{v^N_{Nq}(\zeta)} = \vect{z(f(\zeta)) - D(f(\zeta))}{C(f(\zeta))}
%=
%\vect{\overline{z(\zeta)} - D(f(\zeta))}{C(f(\zeta))},
%\end{equation}
%using (\ref{zf}), and build the rest of $v^N(\zeta)$ recursively.
%Note that for every $n> Nq$, \[{v^N_{n+q}(\zeta)} =
%\overline{z(\zeta)}{v^N_{n}(\zeta)}.\]
%
For $E\in I$ we can define $u^N(E)$ as a limit of $u^N(\zeta)$ when
$\zeta = E + i\epsilon$,\ $\epsilon\searrow0$, as follows: first
define it for $u^N_{ Nq}$ and $u^N_{ Nq+1}$ using (\ref{udef}), then
for any $n$ using the fact that $u^N_n(\cdot)$ is a linear
combination of $u^N_{ Nq}(\cdot)$ and $u^N_{ Nq+1}(\cdot)$. From
(\ref{greensys}), one can see that $u^N(E)$ and
\begin{equation}
\label{vdef}
v^N(E) \df \overline{u^N(E)}
\end{equation}
both satisfy
\[
\an{n}^Nu_{n + 1}+(\V n^N - E)u_n + \an{n-1}^Nu_{n - 1} = 0,
\;\mbox {for every $n\geq 1$}.
\]
%From (\ref{udef}),%(\ref{vdef}),  and (\ref{fe}),
%\begin{equation}
%\label{uv}
%u^N(E) = \overline{v^N(E)}\; \mbox{ for any }E\in I.
%\end{equation}

Comparing the Wronskian (see, e.g.,  \cite[p. 20]{Teschl} for its
definition and properties) of $u^N(E)$ and $v^N(E)$ at zero and at
$Nq$, we get
\begin{equation}
\label{calc}\begin{array}{lc}
a_0\p   (u^N_0     (E)v^N_1   (E) - u^N_1     (E)v^N_0     (E))&=\\
a_{Nq}\p(u^N_{Nq-1}(E)v^N_{Nq}(E) - u^N_{Nq}  (E)v^N_{Nq-1}(E))&=\\
a_{0}\p(C(E)(\overline{z(E)}-D(E)) -(z(E)-D(E))C(E))          &=\\
a_{0}\p C(E)(\overline{z(E)}-z(E)),&
\end{array}
\end{equation}
where the first equality follows from the properties of the
Wronskian,  and the second equality is obtained by combining
(\ref{udef}) and (\ref{vdef}). Hence %by (\ref{uv})
\begin{equation}\label{uz}
\im(u^N_0(E)\overline{u^N_1(E)}) = -C(E)\im(z(E)).
\end{equation}
Also, (\ref{calc}) shows that on $I$ the value of $u^N_0(E)$  is not zero:  indeed, $(z(E)-\overline{z(E)})C(E)\neq0$ on $I$, and $v^N_0(E)=\overline{u^N_0(E)}$.

Since from (\ref{Gu})
\[
\overline{u^N_0(E)}u^N_1(E) =
-\ap0|u^N_0(E)|^2G^N(1, 1; E+i0),
\]
we get from (\ref{green}) and (\ref{uz}) the key formula
\begin{equation}
\label{final}
\frac{d\mu^N_{ac}}{dE}(E) =
\frac{ |C(E)\im(z(E))|}{\pi |a_0\p| | u^N_0(E)|^2},
\end{equation}
where $u^N$ is the solution of (\ref{greensys}) defined by the condition (\ref{udef}).

Now we can see that a uniform lower bound for the integral
\[ \int_I\ln\frac{d\mu^N_{\ac}}{dE}(E)\;dE \] will
follow from a (uniform) upper bound on $\int_I\ln|u^N_0(E)|\;{d}E$,
which is the subject of the next section.

\section{The estimate on $u_0^N$}
In order to obtain an estimate on $u_0^N$ we use the following theorem from  \cite{Den}:
\begin{theorem}[{ \cite[Theorem 2.1]{Den}}]\label{denisov_weight_to_diagonal}
Let
\[
\Psi_{n} = (\mI + \mW_n)\Lambda_n\Psi_{n-1}\!,\;\;
\Psi_0 = \vect10\!,\]\[\Lambda_n
=\matr {\lambda_n}00{\lambda_{n}^{-1}}\!, \;\;
\mW_n = \matr {\alpha_n} {\beta_n} {\gamma_n} {\delta_{n}}\!,
\]
where $\lambda_n \in\C$ and, for some constants $C$, $\kappa$ and $v\in[0,1)$,
 \[C > |\lambda_n| > |\kappa|> 1,\;
\left\|\left\{\left\|\mW_n\right\|\right\}\right\|_2 ( \,\, =
\sqrt{\sum \|W_n\|^2} \,\,) \leq \widetilde{C}, \mbox { and }\]
\begin{equation}\label{v_alpha}
\left|\ln{\prod_{n = k}^l |1 +\alpha_n|}\right| \leq C+v\sqrt{l - k},\;
\left|\ln{\prod_{n = k}^l |1 +\delta_n|}\right| \leq C+v\sqrt{l - k}
\end{equation}
\oo{}where $\widetilde{C} > 0$ is a universal constant.\cc  Then
 \[\Psi_{n} = \left(\prod_{j = 1}^{n }\lambda_j(1 + \alpha_j)
\right)\vect {\phi_n}{\nu_n}, \]
where $\phi_n, \nu_n$ satisfy the following estimate, uniformly in $n$ with some
constant $B>0$ \oo{}depending only on $C$\cc:
\begin{equation}\label{estim1}
|\phi_n|, |\nu_n| \leq B
\exp\left (\frac{B}{|\kappa| - 1}
    \exp\left(\frac{Bv^2}{|\kappa| - 1}
    \right)\right).
\end{equation}
Moreover,
\begin{equation}\label{estim2}
|\phi_n| > 1/\widetilde{B}(\kappa),\;|\nu_n| < \widetilde{B}(\kappa) \left\|\left\{\left\|\mW_n\right\|\right\}\right\|_2
\end{equation}
uniformly in $n$, where $\widetilde{B}(\kappa) > 0$ is a constant depending on $\kappa$ \oo{}and tending
to infinity as $|\kappa|$ approaches 1.\cc
\end{theorem}
We would like to use Theorem \ref{denisov_weight_to_diagonal} with $\Lambda_n$  being
the diagonalization of $\mP_n$. However, $\det \mP_n = a_{nq}/a_{(n+1)q}\neq 1$
because of the perturbation, and we need to ``fix'' the situation. Define,
in the notations of (\ref{tr_period}),
\[M_n = \matr {1}00{a_{nq}}\!,\;\widetilde{\mP}_n = M_n^{-1}\mP_n M_{n + 1}=
\matr{A_n(E)}{a_{(n+1)q}B_n(E)}{a_{nq}\inv C_n(E)}
{\frac{a_{(n+1)q}}{a_{nq}}D_n(E)},
\]
so that $ \det\widetilde{\mP}_n = 1$. Then
\[
\begin{array}{l}
\vect{u_{1}^N(\zeta)}{u_{0}^N(\zeta)} =
    \mP\inv_{0}\mP\inv_{1}\cdots\mP\inv_{N-1}
    \vect{u_{Nq+1}^N(\zeta)}{u_{Nq}^N(\zeta)}
=\\
    M_0\inv M_0 \mP\inv_{0}M_1\inv M_1\cdots M_{N-1}\inv M_{N-1}\mP\inv_{N-1}M_N\inv M_N
    \vect{x(\zeta)}{y(\zeta)}
=\\
    M_0\inv \widetilde{\mP}\inv_{0}\widetilde{\mP}\inv_{1}\cdots\widetilde{\mP}\inv_{N-1}
     M_N
    \vect{x(\zeta)}{y(\zeta)}
=\\M_0\inv
\mU\inv_{0}\mL_{0}\mU_{0}\cdots
    \mU_{N-2}\mU\inv_{N-1}\mL_{N-1}\mU_{N-1}M_N\vect {x(\zeta)}{y(\zeta)},
\end{array}
\]
where
\[
\mU\inv_n(\zeta)\! =\! \matr {\lambda_n(\zeta)\!-\!\frac{a_{nq}}{a_{(n+1)q}}D_n(\zeta)}
{\lambda_n\inv(\zeta)\!-\!\frac{a_{nq}}{a_{(n+1)q}}D_n(\zeta)}{a_{nq} C_n(\zeta)}{a_{nq} C_n(\zeta)}\!.
\]
Recall that, from (\ref{trunc}), we have $a_{(N - 1)q} = a_{Nq} = \ap0$, thence
\[
\mU\inv_{N - 1}(\zeta)\! =\! \matr {z(\zeta)\!-\!D(\zeta)}
{z\inv(\zeta)\!-D(\zeta)}{a_{Nq} C(\zeta)}{a_{Nq} C(\zeta)}\!,
\]
and it is easy to check that $U_{N-1}M_N\vect {x(\zeta)}{y(\zeta)} = \vect{1}{0}$.

Denote $U_{n-1}U_n\inv = \mI + W_n(\zeta)$.
Then
\begin{equation}
\vect{u_{1}^N(\zeta)}{u_{0}^N(\zeta)} =M_0\inv
\mU\inv_{0}\mL_{0}(\mI+\mW_{1})\mL_{1}\cdots(\mI+\mW_{N-1}) \mL_{N-1}\!\vect10\! .
\label{u01}
\end{equation}

Let us denote $\xi_n = (a_{nq}, \ldots, a_{(n+1)q}, \V{nq+1},
\ldots, \V{(n+1)q})\in \R^{2q+1}$. Then $A_n(\zeta)$, $B_n(\zeta)$,
$C_n(\zeta)$ and $D_n(\zeta)$ are polynomials with coefficients
which are analytic functions of $\xi_n$ (recall that $\inf_n a_n
>0$). Hence, by the Mean Value Theorem, the norms of $\mW_{n}$ are
square-summable (for any $\zeta$) and
$\left\|\left\{\left\|\mW_n(\zeta)\right\|\right\}\right\|_2 \leq C
< \infty$, uniformly on $I_{\epsilon_I}\cup I$. By means of a
finitely supported perturbation we can make
$\left\|\left\{\left\|\mW_n(\zeta)\right\|\right\}\right\|_2$ as
small as needed without changing the a.c.\ spectrum of $\J $.

By Proposition \ref{lbound} from the Appendix
\begin{equation}
\label{lambdan} |\lambda_n(\zeta)| > 1+C_I\im\omega \quad \text{on
$I_{\epsilon_I}$}
\end{equation}
for a suitably chosen positive constant $C_I$. We also have for
$\zeta\in I_{\epsilon_I}\cup I$ and some constant $B$
\begin{equation}\label{estim_for_diagonal}\begin{split}
\left|\ln{\prod_{n = k}^l |1 + \alpha_n|}\right|& \leq
B + B\im \zeta\sqrt{l - k},\\
\left|\ln{\prod_{n = k}^l |1 + \delta_n|}\right|& \leq
B+B\im\omega\sqrt{l - k},
\end{split}
\end{equation}
which is proven following the argument in  \cite[Theorem 2.2]{Den}.

Applying Theorem \ref{denisov_weight_to_diagonal} to (\ref{u01}), (\ref{lambdan})
and (\ref{estim_for_diagonal}), we obtain
\[
\vect{u_{1}^N(\zeta)}{u_{0}^N(\zeta)} =\left(\prod_{j = 1}^{N }\lambda_j(1 + \alpha_j) \right)M_0\inv\mU\inv_{0}\mL_{0}\vect {\phi_N}{\nu_N},
\]
that is,
\begin{equation}
\label{umatr}
 u_{0}^N(\zeta) = \prod_{j = 1}^{N }\lambda_j(\zeta)(1 + \alpha_j(\zeta)) \;
    C_0(\zeta)\!\left(\lambda_0(\zeta)\phi_N(\zeta)+ \lambda_0\inv\nu_N(\zeta)\right),
\end{equation}
where
\begin{equation}
\label{estim3}
\begin{split}
|\phi_N|, |\nu_N| &\leq B
\exp\left (\frac{B}{\im \zeta}
    \exp\left(\frac{B(\im \zeta)^2}{\im \zeta}
    \right)\right) \\ &\leq
    B \exp\left (\frac{ B' }{\im \zeta}\right),
\end{split}
\end{equation}
and by means of a finitely supported perturbation we can ensure,
for any fixed $\epsilon>0$, that
\begin{equation}\label{phinubound}
|\phi_N| > 1/B'', \quad |\nu_N| < \epsilon, \quad \text{for} \quad \im \zeta > \frac{\epsilon_I}{2}.
\end{equation}

\section{Proof of the Main Theorem}

From Weyl's theorem we know that $\sigma_{\mathrm{ess}}(\J ) = \sigma_{\mathrm{ess}}(\J \p)$.
We show that the a.c.\ spectrum of $\J $ fills $\spec{\J \p}$ by proving that for the chosen
above (arbitrary) interval $I$:
\[
\int_I{\ln \frac{d\mu_{\ac}}{dE}(E)\; dE} >-\infty.
\]

From (\ref{final}),
\[ \ln \frac{d \mu_\mathrm{ac}^N}{d E}(E)
    = \ln \frac{|C(E)| |\im z(E)|}{\pi |a_0\p| |u_0^N(E)|^2}
    = \ln \frac{|C(E)| |\im z(E)|}{\pi |a_0\p|} - 2 \ln |u_0^N(E)|. \]
Now, (\ref{umatr}) yields
\[\begin{split}
 - \ln |u_0^N(\zeta)|
    &= \\ - \ln &\left|  \prod_{j = 1}^{N }\lambda_j(\zeta)(1 + \alpha_j(\zeta)) \;
        C_0(\zeta)\!\left(\lambda_0(\zeta)\phi_N(\zeta)+ \lambda_0\inv\nu_N(\zeta)\right) \right|= \\
       - \ln& \left|  \prod_{j = 1}^{N }\lambda_j(\zeta)(1 + \alpha_j(\zeta)) \right|
        - \ln \left| C_0(\zeta)
            \!\left(\lambda_0(\zeta)\phi_N(\zeta)+ \lambda_0\inv\nu_N(\zeta)\right) \right|.
\end{split}\]
Denote
\[ f_N(\zeta) = - \ln \left| C_0(\zeta)
            \!\left(\lambda_0(\zeta)\phi_N(\zeta)+ \lambda_0\inv\nu_N(\zeta)\right) \right|;\]
let us show that $f_N$ satisfies the assumptions of Lemma~\ref{denisov_harmonic_lemma}. Obviously,
$f_N$ is harmonic in $I_{\epsilon_I}$ and continuous in $I_{\epsilon_I} \cup I$. Then,
$I_{\epsilon_I} \cup I$ is a compact set, therefore
\[ |C_0(\zeta)| \max(|\lambda_0(\zeta)|, |\lambda_0^{-1}(\zeta)|) \leq B_1, \]
and by (\ref{estim3})
\[\begin{split}
&\left| C_0(\zeta) \!\left(\lambda_0(\zeta)\phi_N(\zeta)+ \lambda_0\inv\nu_N(\zeta)\right) \right| \\
&\qquad\leq B_1 (|\phi_N(\zeta)| + |\nu_N(\zeta)|)
\leq B_2 \exp \left( \frac{B'}{\im \zeta} \right)
    < \exp \left( \frac{B_3}{\im \zeta} \right).
\end{split}\]
Hence
\[ f_N(\zeta) > - \frac{B_3}{\im \zeta} \]
for all $\zeta \in I_{\epsilon_I}$, and the second condition of Lemma~\ref{denisov_harmonic_lemma}
is satisfied.

For
$ \zeta \in \Omega_I = \Big\{ \frac{3}{4}  \epsilon_{I} \leq \im \zeta \leq \epsilon_{I},
    \re \zeta \in I \Big\} $,
\begin{equation}\label{lbfn}\begin{split}
&\left| C_0(\zeta) \!\left(\lambda_0(\zeta)\phi_N(\zeta)+ \lambda_0\inv\nu_N(\zeta)\right) \right| \\
&\qquad\geq \frac{1}{B_4}
    \left[ |\lambda_0(\zeta)| |\phi_N(\zeta)| - |\lambda_0^{-1}(\zeta)||\nu_N(\zeta)| \right].
\end{split}\end{equation}
According to Proposition~\ref{lbound} in the Appendix,
\[ |\lambda_0(\zeta)| \geq 1 + C_I \im \zeta, \quad |\lambda_0^{-1}(\zeta)| \leq 1 - C_I \im \zeta,\]
and in particular
\[ |\lambda_0| \geq \kappa > 1 \quad \text{on} \quad \Omega_I \]
for $\kappa = 1 + \frac{3}{4} C_I \epsilon_I$.
By (\ref{phinubound}),
\[ (\ref{lbfn}) \geq \frac{1}{B_4} \left(\frac{1}{B''} - \epsilon\right) \geq \frac{1}{B_5} > 0 \]
if we choose $\epsilon = \frac{1}{2B''}$. This validates the third
condition of Lemma \ref{denisov_harmonic_lemma}.

On the real line,
\[ f_N(E) = \ln \frac{d\mu^N_\textrm{ac}}{dE}(E) +
    \ln  \left|  \prod_{j = 1}^{N }\lambda_j(E)(1 + \alpha_j(E)) \right|
    - \ln \frac{|C(E)| |\im z(E)|}{\pi |a_0\p|}. \]
By (\ref{estim_for_diagonal}) the second addend is uniformly bounded on $I$; the third
addend is also uniformly bounded, since $I$ is compact. Thus
\[ f_N(E) \leq \ln \frac{d\mu^N_\textrm{ac}}{dE}(E) + B_6, \]
and hence
\[ f_N^+(E) \leq \left( \ln \frac{d\mu^N_\textrm{ac}}{dE}(E) \right)^+ + B_6 \leq
     \frac{d\mu^N_\textrm{ac}}{dE}(E) + B_6.\]
Integrating over $I$, we obtain:
\[ \int_I f_N^+(E) dE \leq B_7, \]
and the first condition of Lemma \ref{denisov_harmonic_lemma} is
also fulfilled. Therefore,
\[ \int_I f_N^-(E) dE \leq B_8, \]
and
\[ \int_I \ln \frac{d\mu^N_\textrm{ac}}{dE}(E) dE
    \geq  \int_I f_N(E) dE - B_9
    \geq  -\int_I f_N^-(E) dE - B_9 \geq  -B_{10}, \]
uniformly in $N$.
Since, as $N\rightarrow\infty$, the measures $\mu^N$ converge weakly to the spectral measure $\mu$ for $\J $, by Lemma \ref{entropy},
\[ \int_I\ln\frac{d\mu_{ac}}{dE}(E) dE \geq - B_{10} > - \infty. \]
This finishes the proof.

\appendix
\section{}
\begin{prop}
\label{lbound}
Suppose $\mP: \C \longrightarrow SL_2(\C)$ is analytic and $\Delta(\zeta) = \tr \mP(\zeta)$.  Let $I\subset\R$ be a closed interval so that $|\Delta(x)|< 2$ and $\Delta'(x) \neq 0$ for any $x\in I$. Then there exist positive constants $\epsilon_I$ and $C_I$, so that for any $\zeta \in
I_{\epsilon_I} \df
\sett{x+iy}{x\in I, 0<y\leq\epsilon_I}$
the matrix $\mP(\zeta)$ has two different eigenvalues $z^{\pm1}(\zeta)$,  where  $|z(\zeta)| < |z\inv(\zeta)|$ and
\begin{equation}
\label{cI}
|z(\zeta)| < 1 - C_I\im\omega, \;|z\inv(\zeta)| > 1 + C_I\im\omega.
\end{equation}
Moreover, let $\{\mP_n(\zeta)\in SL_2(\C)\}$ be a sequence of analytic matrices converging to $\mP$ for any $\zeta\in I_\epsilon\cup I$. Then the constant $C_I$ can be chosen so that the two eigenvalues of $\mP_n(\zeta)$ satisfy (\ref{cI}) for any $n\geq n_0$ for some $n_0$.
\end{prop}
\begin{proof}
Let
\[
z(\zeta) = \frac{\ddo+\sqrt{\ddo^2 -4}}2
\]
be one of the eigenvalues of $\mP(\zeta)$. Define
$f(x, y) = \ln |z(x+iy)|$.

Since $\mP\in SL_2(\C)$ and $|\ddo| < 2$, we have $f(x, 0) = 0$ for any $x\in I$. Below we prove that $g(x)\df\frac{\partial{f}}{\partial{y}}(x, 0)\neq 0$ for any $x\in I$. The function $g(x)$ is continuous on $I$ and either positive or negative there. Suppose $g(x) > 0$ on $I$. The case of $g(x) < 0$ can be treated similarly.
From the Taylor expansion we get
\[
f(x, y) = f(x, 0) + g(x)y + O(y^2) = g(x)y + O(y^2).
\]
 Since $I$ is compact, there exists $\min_{x\in I}g(x) = 2C_I$. Therefore, there exists $\epsilon_I$ so that for every $0<y \leq \epsilon_I$
\[
f(x, y) =\ln|z(x + iy)| = g(x)y + O(y^2) > C_I y.
\]
To prove that $\frac{\partial{f}}{\partial{y}}(x, 0)\neq 0$ on $I$, calculate:
\[
\frac{\partial{f}}{\partial{y}}= \re \frac{\partial{\ln(z(x + iy))}}{\partial{y}} = \re \left(\frac{i}{z(\zeta)}z'(\zeta)\right) = -\im \left(\frac{1}{z(\zeta)}z'(\zeta)\right)
\]
\[
 = -\im \left(\frac{1}{z(\zeta)}\frac{z(\zeta)\Delta'(\zeta)}{\sqrt{\ddo^2 - 4}}\right) \neq 0
\]
for every $\zeta = x\in I$.

Let $\Delta_n(\zeta) = \tr \mP_n(\zeta)$ and
\[
\lambda_n(\zeta) = \frac{\Delta_n(\zeta)+\sqrt{\Delta_n^2(\zeta) -4}}2.
\]
Note that $\{\lambda_n(\zeta)\}$ is a sequence of analytic functions converging pointwise to $z(\zeta)$ on $I_{\epsilon_I}\cup I$ and uniformly bounded on it. Hence, by Vitali's theorem (see, e.g.,  \cite{hille}), the convergence is uniform and the second statement of the Proposition holds true for our choice of $C_I$.
\end{proof}

\end{document}